\documentclass[oneside,11pt]{amsart}
\usepackage[T1]{fontenc}
\usepackage[latin9]{inputenc}
\usepackage{amsthm}
\usepackage{color}
\usepackage{amssymb}
\usepackage{mathtools}
\usepackage{multicol}

\newcommand{\cC}{\mathcal{C}}

\newcommand{\N}{\mathbb{N}}
\newcommand{\Z}{\mathbb{Z}}

\newcommand{\SG}[1]{\mathbb{S}_{#1}}

\newcommand{\K}{\mathbb{K}}
\newcommand{\LI}{\mathcal{L}}

\newcommand{\Alt}{\mathbb{A}}
\newcommand{\Aut}{\mathrm{Aut}}

\newcommand{\id}{\mathrm{id}}

\newcommand{\defFK}{\mathcal{D}}

\newcommand{\defT}{\mathcal{T}}
\newcommand{\defIA}{\mathcal{K}}
\newcommand{\Rad}{\mathrm{Rad}\,}
\newcommand{\FK}{\mathcal{E}}

\newcommand{\Sym}{\mathbb{S}}

\newcommand{\charK }{\mathrm{char}\,\K }

\makeatletter
\def\imod#1{\allowbreak\mkern0mu({\operator@font mod}\,#1)}
\makeatother

\makeatletter
\numberwithin{equation}{section}
\numberwithin{figure}{section}
\numberwithin{table}{section}
\theoremstyle{plain}
\newtheorem{thm}{Theorem}[section]
\newtheorem{lem}[thm]{Lemma}
\newtheorem{cor}[thm]{Corollary}
\newtheorem{pro}[thm]{Proposition}

\newtheorem{defn}[thm]{Definition}
\newtheorem{problem}{Problem}

\theoremstyle{remark}
\newtheorem{rem}[thm]{Remark}

\makeatother

\begin{document}

\title[PBW deformations of a Fomin--Kirillov algebra]{PBW deformations of a Fomin--Kirillov algebra and other examples}

\author{I. Heckenberger}
\author{L. Vendramin}

\address{I. Heckenberger:
	Philipps-Universit\"at Marburg,
	FB Mathematik und Informatik,
	Hans-Meer\-wein-Stra\ss e,
35032 Marburg, Germany.}
\email{heckenberger@mathematik.uni-marburg.de}

\address{L. Vendramin: IMAS--CONICET and Depto. de Matem\'atica, FCEN, Universidad de Buenos Aires, Pabell\'on 1,
Ciudad Universitaria (1428), Buenos Aires, Argentina.}
\email{lvendramin@dm.uba.ar}

\date{}

\subjclass{Primary 16T05; Secondary 20F55}

\thanks{The second-named author is partially supported by PICT-2014-1376, MATH-AmSud 17MATH-01, ICTP,
ERC advanced grant 320974 
and the Alexander von Humboldt Foundation.}

\begin{abstract}
	We begin the study of PBW deformations of graded algebras relevant to the
	theory of Hopf algebras.  One of our examples is the Fomin--Kirillov algebra
	$\FK_3$. Another one appeared in a paper of Garc\'ia Iglesias and Vay.  As a
	consequence of our methods, we determine when the deformations are
	semi\-simple and we are able to produce PBW bases and polynomial identities
	for these deformations.  

	Keywords: Clifford algebra, Fomin--Kirillov algebra, Hopf algebra, Nichols algebra, PBW
	deformation, polynomial identity.

	Corresponding author: Leandro Vendramin
\end{abstract}

\maketitle
\setcounter{tocdepth}{1}

\section*{Introduction}

Deformations of several algebraic structures have been of great interest in the
last years. Such deformations include group algebras, Lie algebras, Weyl
algebras, rational Cherednik algebras, Hecke algebras and generalizations.  A
deformation of a graded algebra $A$ given by generators $a_1,\dots,a_n$ and
homogeneous relations $r_1,\dots,r_m$ is an algebra $D$ given by generators $a_1,\dots,a_n$
and relations $r_1+t_1,\dots,r_m+t_m$, where each $t_j$ is a possibly non-homogeneous 
element of degree less than
the degree of $r_j$.  

The classical Poincar\'e--Birkhoff--Witt Theorem for Lie algebras
motivates the study of a
particular family of deformations.  A PBW deformation of a graded algebra $A$
is a deformation $D$ of $A$ such that the associated graded algebra of $D$ is
isomorphic to $A$.  PBW deformations have been considered in many different
contexts.  In the case of quadratic algebras, these deformations have been
recently studied in~\cite{MR1620538,MR1881922,MR2102087,MR2177131}. For $N$-homogeneous
algebras they have been studied in \cite{MR2236596,MR2256407}. In the context
of Hopf algebras and their actions, a related class of algebras was studied in
\cite{MR3784828} and in \cite{MR3317849}.
PBW deformations satisfying additional properties
also appear in several papers where the classification of
finite-dimensional pointed Hopf algebras is considered, see for
example~\cite{MR3119229}.

Nichols algebras over non-abelian groups form a particularly interesting family
of graded algebras where not very much is known. In general, they are not
$N$-Koszul and are not even generated by homogeneous relations of constant
degree. Nevertheless they have remarkable Hilbert series which indicates a rich
internal structure.  We refer to~\cite{MR1913436} for an introduction to the
theory of Nichols algebras. At this place we do not want to go into details, since
we will only use their known presentations as algebras by generators and relations.
Many of the crucial results on Nichols algebras over non-abelian groups can
only be recovered by extensive Gr\"obner basis calculations.  Nevertheless,
these algebras are very important in particular since they appear as an
essential tool in the classification of pointed Hopf algebras with non-abelian
coradical~\cite{MR2799090}, in
combinatorics~\cite{MR2209265,MR3552907,MR1667680,MR3459053,MR3177534} and in
mathematical physics~\cite{L17, MR2106930, MR1969778, MR1904729}. 

In this paper we begin the study of PBW deformations of some finite-dimensional
Nichols algebras over non-abelian groups. Our long-term objective is to
understand the structure of finite-dimensional Nichols algebras and Fomin--Kirillov algebras 
by means of PBW deformations.
For that purpose we concentrate first on two small examples: 1) The Fomin--Kirillov
algebra $\FK_3$, and 2) the Nichols algebra associated with the vertices of the
tetrahedron and constant cocycle $-1$.

Some particular deformations of the Fomin--Kirillov algebra $\FK_3$ have been
considered in \cite{MR1969778,MR2106930} and in~\cite{Bastian}. The cohomology
of $\FK_3$ has been recently computed in~\cite{MR3459024}.

We compute all PBW deformations of these Nichols algebras. It turns out that the moduli space
of PBW deformations is an affine space of very small dimension. Moreover, generically
the deformations are semisimple and the non-semisimple locus is determined.
Our result has some applications:
a) we get a PBW basis of the algebras; b) we produce non-trivial
polynomial identities for our Nichols algebras that were not known before. For
example, as a corollary we prove that the Fomin--Kirillov algebra $\FK_3$ satisfies
the \emph{Hall identity} 
\[
	[[x,y]^2,z]=0.
\]
The PBW deformations we find have also been obtained by Garc\'ia Iglesias and
Vay~\cite{MR1904729}.
These algebras are useful to obtain the classification of
some finite-dimensional pointed Hopf algebras with non-abelian coradical.
We provide a simplification in the presentation of the
$72$-dimensional example.

\medskip
The paper is organized as follows. In Section~\ref{section:preliminaries} we
define our algebras as certain PBW deformations of Nichols algebras compatible
with the group action and we prove some basic properties. Then we recall a
basic result on the simplicity of Clifford algebras that will be useful to
study our examples.  In Section~\ref{section:FK3} we study PBW deformations of
the Fomin--Kirillov algebra $\FK_3$. In Theorem~\ref{thm:FK3:semisimple} we
precisely determine when the deformations of $\FK_3$ are semisimple; it is
remarkable that in the cases where the deformation is not semisimple one finds
either the preprojective algebra of type $A_2$ or the coinvariant algebra
appearing in Schubert calculus~\cite{MR2209265}. In
Theorem~\ref{thm:FK3:basis} a PBW basis is constructed and PBW deformations of
$\FK_3$ are classified in Theorem~\ref{thm:defFK3}.  In
Section~\ref{section:dim72} PBW deformations of the $72$-dimensional Nichols
algebra associated with the vertices of the tetrahedron are studied.  Using Ore
extensions we produce a PBW basis for the deformation of this Nichols algebra,
see Proposition~\ref{pro:Ore}, Theorems~\ref{thm:T:basis} and~\ref{thm:defT}.
In Theorem~\ref{thm:T:semisimple} we determine when these deformations are
semisimple. Most of the results on this deformation are based on calculations
related to an intermediate algebra studied in Section~\ref{section:auxiliar}.

\medskip
We end the introduction by formulating three problems. 

\begin{problem}
	Classify the PBW deformations of the known finite-dimensional
	Nichols algebras over groups. Decide when these algebras are semisimple.
\end{problem}

\begin{problem}
	Classify the PBW deformations of the Fomin--Kirillov algebras and their
	subalgebras appearing in~\cite{MR3552907}. Decide when these deformations are
	semisimple.
\end{problem}

\begin{problem}
	Study the representation theory of the PBW deformations of the algebras in the
	first two problems.
\end{problem}

\subsection*{Acknowledgments}

We thank Pavel Etingof for pointing out the references~\cite{MR2147005,MR2399085}.

\section{Preliminaries}
\label{section:preliminaries}

Throughout this paper we assume that $\K$ is an algebraically closed field of
characteristic $\ne2$. 

\subsection{Deformations}
\label{section:deformations}

Let $\cC $ be a monoidal category. (In our case, $\cC $ will be the category of
$\K G$-modules for some field $\K $ and some finite group $G$.)

\begin{defn}
	Let $A$ be an $\N_0$-graded algebra in $\cC$. A \emph{PBW deformation} of $A$
	in $\cC$ is an
	$\N_0$-filtered algebra $D_A$ in $\cC $ such that $\mathrm{gr}\,D_A\cong A$.
\end{defn}

In our study we usually meet particular families, which will be defined here.

\begin{defn}
	Let $A$ be an $\N_0$-graded algebra over the field $\K$ and let $s\ge 1$ be an integer.
We say that a family $(A(\lambda))_{\lambda \in \K^s}$ is an \emph{affine family of
deformations of $A$} if there are a family $(a_i)_{i\in I}$ of homogeneous generators of $A$,
a family $(r_j)_{j\in J}$ of homogeneous relations of $A$, a basis $(b_l)_{l\in L}$ of $A$
consisting of homogeneous elements, and a family $(c_{jl})_{j\in J,l\in L}\in \K[t_1,\dots,t_s]^{J\times L}$
such that $A(0)=A$ and for any $\lambda \in \K^s$,
$$ A(\lambda) =\K \langle a_i \mid i\in I\rangle /\Big(r_j+\sum_{l\in L}c_{jl}(\lambda )b_l\mid j\in J\Big), $$
and $c_{jl}=0$ whenever $\deg (b_l)\ge \deg (r_j)$.
\end{defn}

Clearly, being an affine family of deformations of an $\N_0$-graded algebra
does not depend on the choice of homogeneous generators and defining relations.

The following proposition is well-known and the proof is elementary. Similar
arguments have been used by Etingof and Rains for example
in~\cite[\S2.2]{MR2399085} and~\cite[Theorem 6.1]{MR2147005}. We refer
to~\cite{MR1299371} for an introduction to the theory of non-commutative
Gr\"ober basis.

\begin{pro}
	\label{pro:folklore}
	Let $A$ be an $\N_0$-graded finite-dimensional algebra over $\K $.
  Let $s,d\in\N$ and
	let $(A(\lambda))_{\lambda\in\K^s}$ be an affine family of deformations of $A$
	such that $\dim A(\lambda)\leq d$ for all $\lambda\in\K^s$. Assume that $\dim
	A(\lambda)=d$ for all $\lambda$ in a Zariski dense subset of $\K^s$. Then
	$A(\lambda)$ is a PBW deformation of $A$ for all $\lambda\in\K^s$.
\end{pro}

\begin{proof}
  Choose a family $(a_i)_{i\in I}$ of homogeneous generators of $A$,
  a family $(r_j)_{j\in J}$ of homogeneous relations of $A$, a basis $(b_l)_{l\in L}$ of $A$
  consisting of homogeneous elements, and a family
  $(c_{jl})_{j\in J,l\in L}\in \K[t_1,\dots,t_s]^{J\times L}$,
  such that $A(0)=A$ and for any $\lambda \in \K^s$,
  $$ A(\lambda )= \K \langle a_i \mid i\in I\rangle 
  /\Big(r_j+\sum_{l\in L}c_{jl}(\lambda )b_l\mid j\in J\Big), $$
  and $c_{jl}=0$ whenever $\deg (b_l)\ge \deg (r_j)$.
  We may assume that
  $(r_j)_{j\in J}$ is a Gr\"obner basis of $A$. Then, by assumption,
  \begin{align} \label{eq:genGroebner}
    \Big(r_j+\sum_{l\in L}c_{jl}(\lambda)b_l\Big)_{j\in J}
  \end{align}
  is a Gr\"obner basis of $A(\lambda)$ for any $\lambda$ in a Zariski dense subset
  $\Lambda $ of $\K^s$.
  Let $p$ be an $S$-polynomial. Then $p$ reduces with respect to the family
  \eqref{eq:genGroebner} to a polynomial $\sum_{l\in L}c_l(\lambda)b_l$,
  where $c_l\in \K[t_1,\dots,t_s]$ and each restriction $c_l|\lambda$ is zero. 
  It follows that
  $c_l=0$ for all $l\in L$, and hence for any $\lambda \in \K^s$,
  \eqref{eq:genGroebner} is a Gr\"obner basis of $A(\lambda)$. This implies the claim.
\end{proof}

\subsection{Clifford algebras}

Let 
$V$ be a finite-dimensional vector space with basis $v_1,\dots,v_n$
and let $q$ be a quadratic form on $V$. The pair $(V,q)$ is called
a \emph{quadratic space}. The \emph{Clifford algebra} $C(V,q)$ is the algebra
given by generators $v_1,\dots,v_n$ and relations
\[
v_i^2=q(v_i),\quad
v_jv_k+v_kv_j=2q_{jk},
\]
for $1\leq i,j,k\leq n$ with $j<k$, where
$q_{jk}=\frac12(q(v_j+v_k)-q(v_j)-q(v_k))$. It is known that $\dim C(V,q)=2^n$,
see for example~\cite[\S9.2, Corollary 2.7]{MR770063}.

Recall that the radical of a symmetric bilinear form $B$ on a vector space $V$
is the subspace of elements $v\in V$ with $B(v,w)=0$ for all $w\in V$.

\begin{thm}
	\label{thm:Clifford}
	Let $(V,q)$ be a quadratic space. 
	\begin{enumerate}
		\item The radical of $C(V,q)$ is generated by the radical of the symmetric bilinear
			form $B_q$ associated with $q$. 
		\item If $\dim V$ is even and $q$ is nondegenerate, then $C(V,q)$ is simple.
		\item If $\dim V$ is odd and $q$ is nondegenerate, then $C(V,q)$ is the
			product of two simple ideals of dimension $2^{\dim V-1}$ each. 
	\end{enumerate}
\end{thm}

\begin{proof}
	The claims (2) and (3) follow from~\cite[\S9.2, Theorem 2.10]{MR770063}.
	Regarding (1), note that for any element $v$ in the radical $R$ of $B_q$, the
	left ideal generated by $v$ is a nilpotent two-sided ideal of $C(V,q)$. Hence 
	\[
		C(V,q)R\subseteq\Rad C(V,q). 
	\]
	Let $q'$ be the quadratic form on $V/R$ induced by $q$. Then $C(V/R,q')$ is
	semisimple by (2) and (3) and $C(V,q)/C(V,q)R\simeq C(V/R,q')$. Hence $\Rad
	C(V,q)=C(V,q)R$.
\end{proof}

\section{The Fomin--Kirillov algebra $\FK_3$}
\label{section:FK3}

The Fomin--Kirillov algebra $\FK_3$ is defined by generators $a,b,c$ and
relations
\begin{align*}
		&a^2=b^2=c^2=0,\\
 		&ca + bc + ab=cb + ba + ac=0.
\end{align*}
It is known that $\dim\FK_3=12$. A basis is given by
\begin{align*}
	1,
	a, 
	b, 
	c, 
	ab, 
	ac, 
	ba, 
	bc, 
	aba, 
	abc, 
	bac, 
	abac.
\end{align*}
We put this algebra into a different context without using this information. 

\begin{rem}
	It is known that $\FK_3$ together with an appropriate comultiplication,
        counit and antipode is a Nichols algebra. This was first proved by
	Milinski and Schneider~\cite{MR1800714}. The primitive elements
        of this Nichols algebra are spanned by $a,b$ and $c$.
\end{rem}

The Fomin--Kirillov algebra $\FK_3$ first appeared in~\cite{MR1667680} to
provide a combinatorial tool to study the structure constants of Schubert
polynomials, and  independently in the paper~\cite{MR1800714} of Milinski and
Schneider, where pointed Hopf algebras with non-abelian coradical were studied.
It also appeared in the work of Majid and Raineri~\cite{MR1969778}, where
applications to physics where considered.  The cohomology of $\FK_3$ was
computed by \c{S}tefan and Vay in~\cite{MR3459024}. 

\begin{defn}
	\label{def:defFK3}
	For any $\alpha_1,\alpha_2\in\K$ let $\defFK_3(\alpha_1,\alpha_2)$ be the
	deformation of $\FK_3$ given by generators $a,b,c$ and relations 
	\begin{align*}
		&a^2-\alpha_1=b^2-\alpha_1=c^2 - \alpha_1=0,\\
 		&ca + bc + ab - \alpha_2=cb + ba + ac - \alpha_2=0.
	\end{align*}
\end{defn}

Note that the algebras $\defFK_3(\alpha_1,\alpha_2)$ are deformations of $\FK_3$
by definition. However, it is not a priori clear that they are PBW deformations.

For the rest of the section let $\alpha_1,\alpha_2\in \K$.

\begin{rem}
	A direct calculation shows that a Gr\"obner basis for the defining
        ideal of $\defFK_3(\alpha_1,\alpha_2)$
	is given by 
	\begin{align*}
	  &a^2 - \alpha_1=0, && b^2 - \alpha_1=0, && ca + bc + ab - \alpha_2=0,\\
	  &cb + ba + ac - \alpha_2=0, && c^2 - \alpha_1=0, && bab -aba - \alpha_2b +\alpha_2a=0.
	\end{align*}
	We will not use this Gr\"obner basis for our arguments.
\end{rem}

\begin{rem}
	\label{rem:FK3:action}
	The algebra $\defFK_3(\alpha_1,\alpha_2)$ is
	an $\Sym_3$-module algebra, where
	\begin{align*}
		&(12)\cdot a=-b, && (12)\cdot b=-a, && (12)\cdot c=-c,\\
		&(23)\cdot a=-a, && (23)\cdot b=-c, && (23)\cdot c=-b.
	\end{align*}
\end{rem}

\begin{lem} 
	\label{lem:FK3:urels}
	Let $u=a-b$, $v=b-c$ and $w=c-a$ in $\defFK_3(\alpha_1,\alpha_2)$. 
	Then 
	\begin{align*}
		ua&=-bu, & ub &=-au, & uc &=(c-a-b)u,\\
		uv &= vu, &uw &=wu, &vw &= wv.
	\end{align*}
	Moreover 
	\begin{align*}
		uv+vw+uw &= \alpha_2-3\alpha_1, 
		& u^3&=(3\alpha_1-\alpha_2)u, & u^2v+uv^2&=0, & uvw&=0. 
	\end{align*}
\end{lem}

\begin{proof}
	The first formula is obtained as follows:
	\begin{gather*}
		ua=(a-b)a=a^2-ba=b^2-ba=b(b-a)=-bu.
	\end{gather*}
	By acting with $(12)\in\Sym_3$ we obtain that $ub=-au$. Now
	\begin{align*}
    (c-a-b)u&=(c-a-b)(a-b)=ca-cb-\alpha_1+ab-ba+\alpha_1\\
		&=ca+ab-(cb+ba) =-bc+\alpha_2+ac-\alpha_2=(a-b)c=uc.
	\end{align*}
	We conclude that $uv=vu$ and $uw=wu$.
        By acting with $(23)$ on $uv=vu$ we obtain that $vw=wv$.

	From the definitions of $u$, $v$ and $w$ it follows that
	\[
		uv+vw+uw=vw-u^2=\alpha_2-3\alpha_1.
	\]
	From $cb+ba+ac-\alpha_2=0$ one obtains that 
	\[
		bcb+\alpha_1a+bac-\alpha_2b=0=\alpha_1c+bab+acb-\alpha_2b.
	\]
	Hence a direct calculations shows that 
	\begin{align*}
		u^2v+uv^2 &= \alpha_1a-\alpha_1c-bab+bac-acb+bcb=-\alpha_2b+\alpha_2b=0.
	\end{align*}
	Moreover 
	\[
		uvw=uv(-u-v)=-u^2v-uv^2=0.
	\]
	Finally 
	\begin {align*}
	u^3=-u^2(v+w)=-u(uv+uw+vw)+uvw=(3\alpha_1-\alpha_2)u.
	\end{align*}
	This completes the proof.
\end{proof}

\begin{lem}
	\label{lem:FK3:idempotents}
	Assume that $3\alpha_1-\alpha_2\ne0$. Then the elements
	\begin{align*}
		&e_1=\frac{1}{3\alpha_1-\alpha_2}((b+c)^2-(\alpha_1+\alpha_2))=\frac{1}{3\alpha_1-\alpha_2}(c-a)(b-a),\\
		&e_2=\frac{1}{3\alpha_1-\alpha_2}((a+c)^2-(\alpha_1+\alpha_2))=\frac{1}{3\alpha_1-\alpha_2}(c-b)(a-b),\\
		&e_3=\frac{1}{3\alpha_1-\alpha_2}((a+b)^2-(\alpha_1+\alpha_2))=\frac{1}{3\alpha_1-\alpha_2}(a-c)(b-c),
	\end{align*}
	form a set of orthogonal central idempotents of
	$\defFK_3(\alpha_1,\alpha_2)$. 
\end{lem}

\begin{proof}
        Using the first formulas for $e_1,e_2$ and $e_3$,
        we conclude from the definition of 
	$\defFK_3(\alpha_1,\alpha_2)$ that $1=e_1+e_2+e_3$.
        A direct calculation using
	Lemma~\ref{lem:FK3:urels} shows that $e_2$ is an idempotent that commutes
	with $a$, $b$ and $c$. Then $e_1=(12)\cdot e_2$ and $e_3=(23)\cdot e_2$ are
	central idempotents. From Lemma~\ref{lem:FK3:urels} it now follows that
	$e_1$, $e_2$ and $e_3$ are orthogonal.
\end{proof}

\begin{lem}
	\label{lem:FK3:Clifford}
	Assume that $3\alpha_1-\alpha_2\ne0$. Let $V$ be a vector space with basis
	$x_1,x_2$ and let $q\colon V\to\K$ be the quadratic form given by 
	\[
		q(\lambda_1x_1+\lambda_2x_2)=\alpha_1\lambda_1^2+(\alpha_2-\alpha_1)\lambda_1\lambda_2+\alpha_1\lambda_2^2.
	\]
	Then $e_3\defFK_3(\alpha_1,\alpha_2)\simeq C(V,q)$.
\end{lem}

\begin{proof}
	Lemma~\ref{lem:FK3:urels} implies that $e_3(a-b)=0$. Then the algebra
	$e_3\defFK_3(\alpha_1,\alpha_2)$ with unit $e_3$ is given by generators
	$e_3a,e_3c$ and relations
	\begin{gather}
		(e_3a)^2=(e_3c)^2=\alpha_1e_3,\\
		e_3ce_3a+e_3ae_3c=(\alpha_2-\alpha_1)e_3,\\
		\label{eq:consequence}(e_3a-e_3c)^2=(3\alpha_1-\alpha_2)e_3.
	\end{gather}
	Since~\eqref{eq:consequence} follows from the other equalities, the lemma
	holds.
\end{proof}

\begin{rem} \label{rem:FK3:degenerateBF}
	The symmetric bilinear form $B$ of the quadratic form $q$ in Lemma~\ref{lem:FK3:Clifford}
	satisfies
	$$ B(x_1,x_1)=B(x_2,x_2)=\alpha_1,\quad B(x_1,x_2)=\frac 1 2
	(\alpha_2-\alpha_1).$$
	It follows that $B$ is degenerate if and only if $\alpha_1^2-\frac 1
	4(\alpha_2-\alpha_1)^2=0$, that is,
	$(3\alpha_1-\alpha_2)(\alpha_1+\alpha_2)=0$.
\end{rem}

\begin{pro}
	\label{pro:FK3:simple}
	Assume that $(3\alpha_1-\alpha_2)(\alpha_1+\alpha_2)\ne0$.  Then
	$e_1\defFK_3(\alpha_1,\alpha_2)$,  $e_2\defFK_3(\alpha_1,\alpha_2)$ and
	$e_3\defFK_3(\alpha_1,\alpha_2)$ are simple algebras isomorphic to
	$\K^{2\times2}$.
\end{pro}

\begin{proof}
	Using the group action one proves that these algebras are isomorphic.  So it
	suffices to prove that $e_3\defFK_3(\alpha_1,\alpha_2)$ is a simple algebra.
	By Lemma~\ref{lem:FK3:Clifford}, the latter algebra is isomorphic to a
	Clifford algebra. Thus the simplicity of $e_3\defFK_3(\alpha_1,\alpha_2)$
	follows from and Theorem~\ref{thm:Clifford} and
	Remark~\ref{rem:FK3:degenerateBF}.
\end{proof}

Next we discuss the deformation $\defFK_3(\alpha_1,3\alpha_1)$.

\begin{pro}
	\label{pro:FK3:nilpotent}
	Let $\LI$ denote the left ideal of $\defFK_3(\alpha_1,3\alpha_1)$ generated by
	$a-b$ and $b-c$. Then $\LI$ is a two-sided nilpotent ideal of
	$\defFK_3(\alpha_1,3\alpha_1)$. The quotient algebra
	$\defFK_3(\alpha_1,3\alpha_1)/\LI$ has dimension $2$ and is isomorphic to
	the algebra $\K[a]/(a^2-\alpha_1)$.
\end{pro}

\begin{proof}
	Let $u=a-b$ and $v=b-c$.  Lemma~\ref{lem:FK3:urels} implies that $\LI$ is a
	two-sided ideal. Moreover, $u^3=0$, and by acting with the transposition
	$(13)$ we also obtain that $v^3=0$.  Since $uv=vu$, it follows that $\LI$ is
	nilpotent.

	Adding $u$ and $v$ to the defining ideal of $\defFK_3(\alpha_1,3\alpha_1)$,
	one obtains the ideal $$ (b-a, c-a, a^2-\alpha_1, 3a^2-3\alpha_1). $$ This
	implies the last claim.
\end{proof}

Now we prove the main theorems of this section.

\begin{thm}
	\label{thm:FK3:semisimple}
	The algebra $\defFK_3(\alpha_1,\alpha_2)$ is
	semisimple if and only if 
	\[
	(3\alpha_1-\alpha_2)(\alpha_1+\alpha_2)\ne0.
	\]
	In this case $\defFK_3(\alpha_1,\alpha_2)\simeq\left(\K^{2\times2}\right)^3$. 
\end{thm}

\begin{proof}
	By Proposition~\ref{pro:FK3:nilpotent}, 
	$\defFK_3(\alpha_1,3\alpha_1)$ is not semisimple. So we may assume that
	$3\alpha_1-\alpha_2\ne0$. We decompose
	$\defFK_3(\alpha_1,\alpha_2)$ as  
	\[
		\defFK_3(\alpha_1,\alpha_2)\simeq e_1\defFK_3(\alpha_1,\alpha_2)\oplus e_2\defFK_3(\alpha_1,\alpha_2)\oplus e_3\defFK_3(\alpha_1,\alpha_2),
	\]
	where $e_1,e_2,e_3$ are the central idempotents of
	Lemma~\ref{lem:FK3:idempotents}.  Now Proposition~\ref{pro:FK3:simple}
	implies that $\defFK_3(\alpha_1,\alpha_2)$ is semisimple and has dimension $12$ if
	$\alpha_1+\alpha_2\ne0$. In the case where $\alpha_1+\alpha_2=0$, the
	deformation $\defFK_3(\alpha_1,\alpha_2)$ is not semisimple by
	Lemma~\ref{lem:FK3:Clifford} and Theorem~\ref{thm:Clifford}(1).
\end{proof}

\begin{thm}
	\label{thm:FK3:basis}
	The algebra
	$\defFK_3(\alpha_1,\alpha_2)$ is a PBW deformation of $\FK_3$ and
	\begin{align}
		\label{eq:basisFK3}
		(a-b)^{n_1}a^{n_2}c^{n_3},\quad
		0\leq n_1\leq 2,\;
		0\leq n_2,n_3\leq 1,
	\end{align}
	is a basis of $\defFK_3(\alpha_1,\alpha_2)$.
\end{thm}

\begin{proof}
	Let $u=a-b$.  Consider the lexicographic ordering on the words in the letters
	$u$, $a$ and $c$, induced by $u<a<c$. Definition~\ref{def:defFK3} and
	Lemma~\ref{lem:FK3:urels} imply that $u^3=(3\alpha_1-\alpha_2)u$ and
	\begin{align} \label{eq:D3newrels}
    \begin{aligned}
		au &= u^2-ua,\\
		cu &= u^2-2ua+uc,\\
		ca &= \alpha_2-\alpha_1+u^2-ua+uc-ac.
    \end{aligned}
	\end{align}
	Hence the monomials $au,cu,ca$ can be written as linear combinations of
	lexicographically smaller ordered monomials.
	Therefore~\eqref{eq:basisFK3} spans $\defFK_3(\alpha_1,\alpha_2)$.
	By Theorem~\ref{thm:FK3:semisimple} and
	Proposition~\ref{pro:folklore},
	$\defFK_3(\alpha_1,\alpha_2)$ is a PBW deformation of $\FK_3$ of dimension
	$12$ for all $\alpha_1,\alpha_2\in \K$. Thus the $12$ elements in
	\eqref{eq:basisFK3} form a basis of $\defFK_3(\alpha_1,\alpha_2)$.
\end{proof}

\begin{rem}
  (1)
  In the proof of Theorem~\ref{thm:FK3:basis} we identified three quadratic
  relations \eqref{eq:D3newrels} of $\defFK_3(\alpha_1,\alpha_2)$.
  Together with the relations $a^2=c^2=\alpha_1$ these form a set of defining
  relations. Indeed, $\defFK_3(\alpha_1,\alpha_2)$ is defined by three generators
  and five linearly independent quadratic relations for them. Using Gr\"obner
  basis calculations it is possible to check that for any $\alpha_1,\alpha_2\in \K$,
  none of these relations is superfluous.

  (2) {}From the PBW basis of $\FK_3=\defFK_3(0,0)$ in Theorem~\ref{thm:FK3:basis}
  one recovers quickly
  that the Hilbert series of $\FK_3$ is the polynomial $(1+t)^2(1+t+t^2)$.
  Since the Hilbert series of the Fomin-Kirillov algebras $\FK_4$ and $\FK_5$
  have a similar form, we expect that also the latter have a PBW basis.
\end{rem}

\begin{cor}
	\label{cor:FK3:pi}
	Any polynomial identity of $\left(\K^{2\times2}\right)^3$ is a polynomial
	identity of $\defFK_3(\alpha_1,\alpha_2)$. 
\end{cor}

\begin{proof}
	It follows from Theorem~\ref{thm:FK3:semisimple} and an argument similar to
	Proposition~\ref{pro:folklore}.		
\end{proof}

The following result classifies PBW deformations of $\FK_3$. These deformations
already appeared in~\cite[Theorem 6.2]{MR3119229} and were used to obtain the
classification of finite-dimensional pointed Hopf algebras with coradical
isomorphic to $\Sym_3$, see also~\cite{MR2766176}. 

\begin{thm}
	\label{thm:defFK3}
	Each PBW deformation of $\FK_3$ in the
	category of $\SG 3$-modules is of the form $\defFK_3(\gamma_1,\gamma_2)$,
	$\gamma_1,\gamma_2\in\K$.
\end{thm}

\begin{proof}
	Let $\defFK$ be a PBW deformation of $\FK_3$. 
	Theorem~\ref{thm:FK3:basis} for $\alpha_1=\alpha_2=0$
	implies that~\eqref{eq:basisFK3} is a basis of $\defFK$. 
	Since $a^2=0$ in $\FK_3$, there exist $\lambda_1,\dots,\lambda_4\in\K$ such that 
	\[
		a^2+\lambda_1a+\lambda_2b+\lambda_3c+\lambda_4=0
	\]
	in $\defFK$.  By acting with the transposition $(2\,3)$, see
	Remark~\ref{rem:FK3:action},
	one obtains that
	\[
	2\lambda_1a+(\lambda_2+\lambda_3)(b+c)=0
	\]
	and hence $\lambda_1=\lambda_2+\lambda_3=0$. By acting with $(1\,2)$ and
	$(1\,3)$ it follows that 
	\[
		b^2+\lambda_2(c-a)+\lambda_4=0,\quad
		c^2+\lambda_2(a-b)+\lambda_4=0.
	\]
	Since $ab+bc+ca=0$ in $\FK_3$, there exist $\lambda_5,\dots,\lambda_8\in\K$ such that 
	\[
		ab+bc+ca+\lambda_5a+\lambda_6b+\lambda_7c+\lambda_8=0
	\]
	in $\defFK$. 
	By acting with the transpositions $(2\,3)$ and $(1\,3)$ we obtain that 
	\begin{align*}
		&ac+cb+ba-\lambda_5a-\lambda_6c-\lambda_7b+\lambda_8=0,\\
		&cb+ba+ac-\lambda_5c-\lambda_6b-\lambda_7a+\lambda_8=0.
	\end{align*}
	This implies that $\lambda_5=\lambda_6=\lambda_7$. Now the commutator of
        $c^2+\lambda_2(a-b)+\lambda_4$ and $c$ in $\defFK$ becomes
        $\lambda_2(ab-ba)$ up to linear and constant terms and hence $\lambda_2=0$.
        Finally, let
        $x=ab+bc+ca+\lambda_5(a+b+c)+\lambda_8$. Then
        by using other quadratic relations it follows that
        $xa-bx=aba-bab+\lambda_5(ca-bc)$
        up to linear and constant terms.
        Let $y=xa-bx$. Then
        $$y-(1\,2)\cdot y=\lambda_5(ca-bc-cb+ac)$$ 
				up to linear and constant terms. Hence $\lambda_5=0$ and
	$\defFK=\defFK_3(-\lambda_4,-\lambda_8)$.
\end{proof}

We now describe the cases where $\defFK_3(\alpha_1,\alpha_2)$ is not
semisimple.

\begin{pro}
	\label{pro:preprojective}
	Assume that  
	$\alpha_1\ne0$. 
	Then the deformation 
	$\defFK_3(\alpha_1,-\alpha_1)$ is isomorphic to the product of three copies of the 
	preprojective algebra of the Dynkin quiver of type $A_2$.
\end{pro}

\begin{proof}
	It follows directly from Lemma~\ref{lem:FK3:Clifford}. 
\end{proof}

\begin{pro}
	\label{pro:coinvariant}
	Assume that  
	$\alpha_1\ne0$. Then the deformation $\defFK_3(\alpha_1,3\alpha_1)$ can be
	presented as a quiver with relations in the following way: The quiver has two
	vertices $1$ and $2$, there are two arrows $ u_{12} , v_{12}$ from $1$ to $2$ and two
	arrows $ u_{21}, v_{21}$ from $2$ to $1$.  The relations are those of the coinvariant
	ring of $\Sym_3$, i.e. 
	\begin{align*}
		u_{ij} v_{ji} = v_{ij}u_{ji}, &&
		u_{ij} v_{ji}+ v_{ij}  w_{ji}+ u_{ij} w_{ji}=0, && 
		u_{ij}  v_{ji} w_{ij}=0
	 \end{align*}
	for all $i,j\in\{1,2\}$ with $i+j=3$, 
	where $ w_{kl}=- u_{kl}- v_{kl}$ for all $k,l\in\{1,2\}$ with $k+l=3$.
\end{pro}

\begin{proof}
	By Theorem~\ref{thm:FK3:basis}, $\dim\defFK_3(\alpha_1,\alpha_2)=12$. Let
	\[
	f_1=\frac{\sqrt{\alpha_1}+b}{2\sqrt{\alpha_1}},\quad 
	f_2=\frac{\sqrt{\alpha_1}-b}{2\sqrt{\alpha_1}}. 
	\]
	Then $f_1$ and $f_2$ are primitive idempotents. (The primitivity follows from
	the fact that the quotient of the deformation by the radical is
	$2$-dimensional, see Proposition~\ref{pro:FK3:nilpotent}.) The action of the
	transposition $(13)$ permutes both $f_1$, $f_2$ and $u=a-b$, $v=b-c$. Let
	$w=-u-v$. By Lemma~\ref{lem:FK3:urels}, the elements $u$, $v$ and $w$
	pairwise commute and
	\[
		u+v+w=0,
		\quad
		uv+vw+uw=0,
		\quad
		uvw=0.
	\]
	Using the equations 
	\[
		f_1u=uf_2-\frac{u^2}{2\sqrt{\alpha_1}},
		\quad
		f_2v=vf_1-\frac{v^2}{2\sqrt{\alpha_1}},
	\]
	one can show that the elements $f_i$, $f_iu^2f_i$ and $f_iv^2f_i$ span
	$f_i\defFK_3(\alpha_1,3\alpha_1) f_i$ for $i\in \{1,2\}$, and that
	$u_{ij}=f_iuf_j$, $v_{ij}=f_ivf_j$ and $f_iu^2vf_j$ for $i,j\in\{1,2\}$ and
	$i\ne j$ span $f_i\defFK_3(\alpha_1,3\alpha_1) f_j$. Then one shows that
	$f_i$, $u_{ij}$ and $v_{ij}$, where $i,j\in\{1,2\}$ with $i\ne j$, generate
	the algebra $\defFK_3(\alpha_1,3\alpha_1)$ and satisfy the relations in the
	proposition.
\end{proof}

\section{An intermediate algebra}
\label{section:auxiliar}

In this section we assume that $\K$ is an algebraically closed
field of characteristic different from
$2$ and $3$. Let $\alpha_1,\alpha_2\in \K$ and let $\beta=3\alpha_1-\alpha_2$.

\begin{defn}
	Let
	$\defIA(\alpha_1,\alpha_2)$ be the associative $\K$-algebra given by
	generators $a,b,c,y$ and relations
	\begin{align*}
	  a^2=b^2=c^2=\alpha_1,&& ab+bc+ca=\alpha_2,&& 
		ac+cb+ba=\alpha_2+y. 
	\end{align*}
  For any $\alpha_3\in \K$ let
  $\defIA(\alpha_1,\alpha_2,\alpha_3)=\defIA(\alpha_1,\alpha_2)/(y^3-\alpha_3)$.
\end{defn}

\begin{lem} 
	\label{le:abcyrels}
  In the algebra
	$\defIA(\alpha_1,\alpha_2)$ the following relations hold.
	\begin{gather*}
		ya=cy,\quad yb=ay, \quad yc=by,\\
		bab-aba=\alpha_2(b-a),\quad (b-a)^3=\beta(b-a).
	\end{gather*}
	In particular, $y^3$ is a central element of $\defIA(\alpha_1,\alpha_2)$.
\end{lem}

\begin{proof}
	First we obtain that
	\begin{align*}
		0=&b(ab+bc+ca-\alpha_2)-(ab+bc+ca-\alpha_2)a\\
		=&bab+\alpha_1c+bca-\alpha_2b-aba-bca-\alpha_1c+\alpha_2a\\
		=&bab-aba-\alpha_2(b-a).
	\end{align*}
	This implies directly the last two equations of the lemma.
        Now we conclude that
	\begin{align*}
		yb-ay=&(ac+cb+ba-\alpha_2)b-a(ac+cb+ba-\alpha_2)\\
		=&acb+\alpha_1c+bab-\alpha_2b-\alpha_1c-acb-aba+\alpha_2a\\
		=&0.
	\end{align*}
	The other two commutation rules for $y$
        follow from this one using that the cyclic
	group $C_3$ acts on
	$\defIA(\alpha_1,\alpha_2,\alpha_3)$ 
	by permuting the generators
	$a,b,c$ cyclically and by fixing $y$. Then it is also clear that $y^3$
        is a central element of $\defIA(\alpha_1,\alpha_2)$.
\end{proof}

\begin{lem}
	\label{lem:u1v1}
	Let $\zeta \in \K$ be such that $\zeta^2+\zeta +1=0$ and let 
	\begin{gather*}
		t=a+b+c,\quad
		v_+=a+\zeta b+\zeta^2c,\quad
		v_-=a+\zeta^{-1}b+\zeta^{-2}c
	\end{gather*}
	be in $\defIA(\alpha_1,\alpha_2)$.  Then
	\begin{align}
		\label{eq:u1v1-1}
		yv_+&=\zeta v_+y, & yv_-&=\zeta^ {-1}v_-y,\\
		\label{eq:u1v1-2}
		v_+v_-&=\beta+\zeta y,& v_-v_+&=\beta+\zeta^2y,\\
		\label{eq:u1v1-3}
		tv_+&=-v_+t-v_-^2,&
		tv_-&=-v_-t-v_+^2.
	\end{align}
\end{lem}

\begin{proof}
	The first two equalities follow from Lemma~\ref{le:abcyrels}.  The proof of
	the other formulas is straightforward from the definitions.
\end{proof}

\begin{lem}
	\label{lem:span}
	Let $\zeta \in \K$ be such that $\zeta^2+\zeta +1=0$. Then 
	the elements $v_+^{n_1}t^{n_2}y^{n_3}$ and $v_-^{n_1}t^{n_2}y^{n_3}$, where
	$n_1,n_2,n_3\in\N_0$, span $\defIA(\alpha_1,\alpha_2)$. 
\end{lem}

\begin{proof}
	Clearly, $v_+$, $v_-$, $t$ and $y$ generate $\defIA(\alpha_1,\alpha_2)$.
	Since $yt=ty$,~\eqref{eq:u1v1-1} and~\eqref{eq:u1v1-3} imply that
	$\defIA(\alpha_1,\alpha_2)$ is spanned by the monomials 
	\[
		v_{s_1}v_{s_2}\cdots v_{s_r}t^{n_2}y^{n_3},\quad
		r,n_2,n_3\in\N_0,\quad 
		s_1,\dots,s_r\in\{-,+\}.
	\]
	Now use~\eqref{eq:u1v1-2} and~\eqref{eq:u1v1-1} to conclude the lemma.
\end{proof}

\begin{lem}
	\label{lem:uv}
	Let $\zeta \in \K$ be such that $\zeta^2+\zeta +1=0$.
	Then the following formulas hold in $\defIA(\alpha_1,\alpha_2)$:
	\begin{align}
		\label{eq:u12}t^2=&y+3\alpha_1+2\alpha_2,\\
		\label{eq:u1v13}tv_+^3+v_+^3t=&2y^2-2\beta y-\beta^2,\\
		\label{eq:v16}v_+^6=&y^3+\beta^3,\\
		\label{eq:vp3=vm3} v_+^3=&v_-^3.
	\end{align}
\end{lem}

\begin{proof}
	The first equality follows from the definitions. Using the last two equations
	in Lemma~\ref{lem:u1v1} we obtain that 
	\begin{align*}
		t&v_+^3+v_+^3t=t(-tv_--v_-t) v_++v_+(-tv_--v_-t)t\\
		=&-t^2v_-v_++\left(v_+^2+v_-t\right)tv_+
		+v_+t\left(v_+^2+tv_-\right)-v_+v_-t^2\\
		=&-t^2v_-v_++v_-t^2v_++v_+t^2v_--v_+v_-t^2-v_+v_-^2v_+.
	\end{align*}
	Then~\eqref{eq:u1v13} follows from~\eqref{eq:u12} and 
	the other
	equations in Lemma~\ref{lem:u1v1}. 
	Using~\eqref{eq:u1v1-3} we obtain that 
	\[
		v_+^3=v_+^2v_+=-tv_-v_+-v_itv_+=-v_-v_+t-v_-tv_+=v_-v_-^2=v_-^3.
	\]
	Finally,
	\begin{align*}
		v_+^6=v_+^3v_-^3
		=(\beta+\zeta y)(\beta+\zeta^2y)(\beta+y)=y^3+\beta^3
	\end{align*}
	because of~\eqref{eq:vp3=vm3},~\eqref{eq:u1v1-1} and~\eqref{eq:u1v1-2}.
\end{proof}

\begin{lem}
	\label{lem:invertible}
	Let $\alpha_3\in\K\setminus\{0\}$. 
	There exists $\lambda\in\K$ such that $v_++\lambda v_-^2$ is invertible in
	$\defIA(\alpha_1,\alpha_2,\alpha_3)$.
\end{lem}

\begin{proof}
	Assume first that $\alpha_3+\beta^3\ne0$. Then $v_+^6$ is a non-zero constant
	by Equation~\eqref{eq:v16} of Lemma~\ref{lem:uv} and hence $\lambda=0$ works.
	Now assume that $\alpha_3+\beta^3=0$.
        Let $\zeta \in \K $ such that $\zeta^2+\zeta+1=0$.
        Since $y^3=-\beta^3$, there exist
	orthogonal idempotents $e_0,e_1,e_2\in\K[y]$ such that $ye_i=-\zeta^i\beta
	e_i$ for all $i\in\{0,1,2\}$ and $e_0+e_1+e_2=1$. Let
	$\lambda\in\K\setminus\{0\}$. Then Lemma~\ref{lem:uv} implies that 
	\begin{align*}
	  (v_++\lambda v_-^2)(v_--\lambda v_+^2)
          &= \beta+\zeta y-\lambda^2(\beta+\zeta^2 y)(\beta+y).
	\end{align*}
	Then
	\[
		e_i(v_++\lambda v_-^2)(v_--\lambda v_+^2)=e_i\beta(1-\zeta^{i+1}-\lambda^2\beta(1-\zeta^i)(1-\zeta^{2+i}))
	\]
	is non-zero for all $i\in\{0,1,2\}$. Hence $v_++\lambda v_-^2$ is invertible. 
\end{proof}

For the formulation of the next claim we need additional notation.
For any $\gamma \in \K$
let $(V,q_\gamma )$ be a two-dimensional quadratic space
with basis $x_1,x_2$ such that
$$q_\gamma (\lambda_1x_1+\lambda_2x_2)=(\gamma+3\alpha_1+2\alpha_2)\lambda_1^2
+(2\gamma^2-2\beta \gamma -\beta^2)\lambda_1\lambda_2
+(\gamma^3+\beta^3)\lambda_2^2 $$
for all $\lambda_1,\lambda_2\in \K$.
If $\gamma^3+\beta^3\ne 0$, then let $x',x''\in C(V,q_\gamma)$ be
the elements
\begin{equation}
\begin{gathered}
	x'=-x_1-(\beta+\zeta^2\gamma)^{-1}x_2,\\
	x''=x_1+(\beta+\zeta^2\gamma)^{-1}x_2-(\beta+\gamma)^{-1}x_2
\end{gathered}
\end{equation}
and let
$Y,A,B,C\in C(V,q_\gamma)^{3\times 3}$ be the matrices
\begin{align*}
		Y=&\begin{pmatrix} \gamma & 0 & 0\\
			0 & \zeta \gamma & 0\\
			0 & 0 & \zeta^2\gamma
		\end{pmatrix},\\
		A=&\frac 1 3\begin{pmatrix} x_1 & (\beta+\zeta^2\gamma) & x_2\\
		  1 & x' & \beta+\gamma\\
		  (\beta+\zeta\gamma)x_2^{-1} & 1 & x''
		\end{pmatrix},\\
		B=&\frac 1 3\begin{pmatrix} x_1 & \zeta(\beta+\zeta^2\gamma) & \zeta^2x_2\\
		  \zeta^2 & x' & \zeta(\beta+\gamma)\\
		  \zeta (\beta+\zeta\gamma)x_2^{-1} & \zeta^2 & x''
		\end{pmatrix},\\
		C=&\frac 1 3\begin{pmatrix} x_1 & \zeta^2(\beta+\zeta^2\gamma) & \zeta x_2\\
		  \zeta & x' & \zeta^2(\beta+\gamma)\\
		  \zeta^2(\beta+\zeta\gamma)x_2^{-1} & \zeta & x''
		\end{pmatrix}.
	\end{align*}

	\begin{rem} \label{re:qgamma_nondeg}
		The discriminant of the quadratic form $q_\gamma $ above is
		$$ -9(4\alpha_1\gamma^3+\beta^3(\alpha_1+\alpha_2)).$$
		Thus the quadratic space $(V,q_\gamma)$ is nondegenerate if and only if
		this expression is non-zero.
	\end{rem}

\begin{lem} \label{le:rho}
	Let $\gamma ,\zeta \in \K$ with $\zeta^2+\zeta+1=0$. Assume
	that $\gamma ^3+\beta^3\ne 0$.
	Then there exists an algebra map
	$\rho_\gamma : \defIA(\alpha_1,\alpha_2)\to C(V,q_\gamma)^{3\times 3}$ such that
	$$\rho_\gamma (a)=A,\quad \rho_\gamma (b)=B,\quad \rho_\gamma (c)=C,\quad
	\rho_\gamma (y)=Y. $$
	This map also satisfies the identity $\rho_\gamma (y^3)=\gamma^3\id$.
\end{lem}

\begin{proof}
	The matrices $A,B,C,Y$ are well-defined.
	It is straightforward to check that the equations
	\begin{gather*}
		A^2=B^2=C^2=\alpha_1\id,\\
		AB+BC+CA=\alpha_2\id,\quad AC+CB+BA=\alpha_2\id+Y,\quad Y^3=\gamma^3\id
	\end{gather*}
	hold.
\end{proof}

\begin{pro} \label{pro:simpleKmodule}
	Let $\gamma,\zeta\in \K$ such that $\zeta^2+\zeta+1=0$ and $\gamma(\gamma^3+\beta^3)\ne 0$.
	Let $M$ be a simple $C(V,q_\gamma )$-module. Then
	$M^3$ is a simple $\defIA(\alpha_1,\alpha_2,\gamma^3)$-module via
	$$ xm=\rho _\gamma (x)m \quad \text{for all $x\in \defIA(\alpha_1,\alpha_2,\gamma^3)$,
	$m\in M^3$.}
	$$
\end{pro}

\begin{proof}
	Lemma~\ref{le:rho} implies that $M^3$ is a left
	$\defIA(\alpha_1,\alpha_2,\gamma^3)$-module. Let $N$ be a
	non-zero submodule of $M^3$.
	Since $\K$ is algebraically closed, there exists an eigenvector $m\in N$ of
	$\rho_\gamma (y)$.
	Moreover, Equation~\eqref{eq:v16} implies that
	$\rho_\gamma (v_+^6)=(\gamma^3+\beta^3)\id \ne 0$.
	Thus, by \eqref{eq:u1v1-1} we may assume that $ym=\gamma m$. Since $\gamma \ne
	0$, we conclude that $m\in M\times\{0\}\times\{0\}$.
	Now observe that
	$$x_1m'=\rho_\gamma (t)m',\quad x_2m'=\rho_\gamma (v_+^3)m' \quad \text{for
		all $m'\in	M\times\{0\}\times\{0\}$.}
	$$
	Hence the simplicity of $M$ implies that $M\times \{0\}\times \{0\}\subseteq N$.
	Then $N=M$ because of \eqref{eq:u1v1-1}.
\end{proof}

\begin{thm}
	\label{thm:auxiliar}
			Let $\alpha_3\in \K$.
	\begin{enumerate}
		\item
			The elements
    	\begin{align*}
		    (a-b)^{n_1}a^{n_2}c^{n_3}y^{n_4},\quad n_1,n_4\in \{0,1,2\},\,n_2,n_3\in
		    \{0,1\},
	    \end{align*}
	    form a basis of
			$\defIA(\alpha_1,\alpha_2,\alpha_3)$.
	    In particular, $\dim \defIA(\alpha_1,\alpha_2,\alpha_3)=36$.
		\item $\defIA(\alpha_1,\alpha_2,\alpha_3)$ is a PBW deformation of
			$\defIA(0,0,0)$.
	\end{enumerate}
\end{thm}

\begin{proof}
	Similarly to the proof of Theorem~\ref{thm:FK3:basis}
	one shows that the elements in (1) span $\defIA(\alpha_1,\alpha_2,\alpha_3)$.
	In particular, $\dim \defIA(\alpha_1,\alpha_2,\alpha_3)\le 36$.

	Let $\gamma\in \K$ with $\gamma^3=\alpha_3$. Assume that
	\begin{align} \label{eq:simplemodulecondition}
	\gamma(\gamma^3+\beta^3)(4\alpha_1\gamma^3+\beta^3(\alpha_1+\alpha_2))\ne
	0.
\end{align}
  By Remark~\ref{re:qgamma_nondeg}, the quadratic space $(V,q_\gamma)$ is
	nondegenerate. Thus $C(V,q_\gamma)$ is simple by
	Theorem~\ref{thm:Clifford}(2). Hence there exists a $2$-dimensional simple
	$C(V,q_\gamma)$-module. Since $\gamma(\gamma^3+\beta^3)\ne 0$,
	by Proposition~\ref{pro:simpleKmodule} there exists a simple
	$6$-dimensional $\defIA(\alpha_1,\alpha_2,\gamma^3)$-module. Hence
	$\dim \defIA(\alpha_1,\alpha_2,\gamma^3)=36$.
	The variety of all triples $(\alpha_1,\alpha_2,\gamma^3)$ satisfying
	\eqref{eq:simplemodulecondition}
	is a dense subvariety of $\K^3$. Therefore the theorem follows from
	Proposition~\ref{pro:folklore}.
\end{proof}

In order to determine the radical of the deformation
$\defIA(\alpha_1,\alpha_2,\alpha_3)$ the following proposition is useful.

\begin{pro}
	\label{pro:radical}
	Let $A$ be a $\K$-algebra and $x,y\in A$. Let $n\in \N$ and
	$\zeta ,\lambda \in \K\setminus \{0\}$.
	Assume that $\charK $ does not divide $n$, $\zeta $ is a primitive root of $1$
	of order $n$, $x$ is invertible, $yx=\zeta xy$, and $y^n=\lambda $.
	Then there exists a primitive idempotent $e\in \K[y]$ such that the
	following claims hold.
	\begin{enumerate}
		\item $A=\oplus _{i,j=0}^{n-1} x^i eAe x^j$.
		\item The map given by $I\mapsto eIe$ is a bijection
			between the ideals of $A$ and the ideals of $eAe$.
			The inverse map is given by $eIe\mapsto AeIeA$.
		\item $\Rad A=\oplus _{i,j=0}^{n-1} x^i \Rad (eAe) x^j$.
	\end{enumerate}
\end{pro}

\begin{proof}
	Let $\gamma\in \K$ be such that $\gamma^n=\lambda$.
	Since $\charK \nmid n$ and $y^n=\lambda $,
	there exist unique
	idempotents $e=e_0,e_1,\dots,e_{n-1}\in \K[y]$ such that $ye_i=\gamma
	\zeta^i e_i$
	for all $i$. Moreover,
	these idempotents are orthogonal and primitive.
	Using these properties, one checks that
	$e_i=x^{i+kn}ex^{-i-kn}$ for all $i$ and all $k\in \Z$.

	(1) Use that $A=\oplus_{i,j=0}^{n-1}e_iAe_j$ and that
	$e_i=x^iex^{-i}$ for all $i$.

	(2) One has to check that $I=AeIeA$ for all ideals $I$ of $A$ and that
	$eAeJeAe=eJe$ for all ideals $eJe$ of $eAe$. The second claim is obvious. The
	first one follows from
	$$ I=\sum_{i,j=0}^{n-1}e_iIe_j=\sum_{i,j=0}^{n-1}x^iex^{-i}Ix^jex^{-j}
	\subseteq AeIeA.
	$$

	(3) Since $\Rad (eAe)=e(\Rad A)e$, (2) implies that
	$$ \Rad A=A e(\Rad A)e A=\bigoplus_{i,j=0}^{n-1}e_iAe\Rad (eAe)eAe_j. $$
	Now use that $e_i=x^iex^{-i}$ and $e_j=x^{j-n}ex^{n-j}$ for all $i,j\in
	\{1,\dots,n-1\}$ and that $\Rad(eAe)$ is an ideal of $eAe$.
\end{proof}

\begin{lem} \label{lem:eKe}
	Let $\gamma \in \K\setminus \{0\}$ and $e\in \K[y]\subseteq
	\defIA(\alpha_1,\alpha_2,\gamma^3)$ be an idempotent.
	Assume that $ye=\gamma e$.
	Then the algebra $e\defIA(\alpha_1,\alpha_2,\gamma^3)e$ is generated by
	$et$ and $ev_+^3$ and is isomorphic to the Clifford algebra
	$C(V,q_\gamma)$.
\end{lem}

\begin{proof}
	Let $A=\defIA(\alpha_1,\alpha_2,\gamma^3)$ and let $\zeta \in \K$ be such that
	$\zeta^2+\zeta+1=0$. Then $A$ is generated by
	$t,v_+,v_-$ and $y$. Lemma~\ref{lem:u1v1} implies that
	$v_+^{n_1}t^{n_2}y^{n_3}$,
	$v_-^{n_1}t^{n_2}y^{n_3}$, where $n_1,n_2,n_3\in \N_0$, span $A$.

	By Lemma~\ref{lem:invertible}, there exists $\lambda \in \K$ such that
	$x=v_++\lambda v_-^2\in A$ is invertible. Moreover, $yx=\zeta xy$ by
	Lemma~\ref{lem:u1v1}.
	Then Proposition~\ref{pro:radical} implies that
	$A=\oplus_{i,j=0}^2x^ieAex^j$. Note that
	$v_+^{n_1}t^{n_2}y^{n_3}e\in x^{n'_1}eAe$ and
	$v_-^{m_1}t^{n_2}y^{n_3}e\in x^{m'_1}eAe$
	for all $n_1,m_1,n_2,n_3\in \N_0$, where $n'_1,m'_1\in \{0,1,2\}$ such that
	$n_1\equiv n'_1 \pmod 3$ and $m_1\equiv -m'_1\pmod 3$.
	Therefore $eAe$ is generated by $ey=\gamma e$, $et$ and $ev_+^3=ev_-^3$,
	see~\eqref{eq:vp3=vm3}. Moreover, Theorem~\ref{thm:auxiliar}(1) and
	Proposition~\ref{pro:radical}(1) imply that $\dim eAe=4$. Then the claim
	follows from Lemma~\ref{lem:uv}.
\end{proof}

\begin{cor}
	Let $\gamma \in \K$.
	\begin{enumerate}
		\item If $\gamma=0$ then $y$ generates a nilpotent ideal of
			$\defIA(\alpha_1,\alpha_2,\gamma^3)$.
		\item If $\gamma\ne 0$ then
			$\defIA(\alpha_1,\alpha_2,\gamma^3)$ is semisimple if and only if
			$q_\gamma $ is nondegenerate. In this case,
			$\defIA(\alpha_1,\alpha_2,\gamma^3)\simeq \K^{6\times 6}$.
	\end{enumerate}
\end{cor}

\begin{proof}
	Let $A=\defIA(\alpha_1,\alpha_2,\gamma^3)$.

	(1) Lemma~\ref{le:abcyrels} implies that $N=Ay$ is a two-sided ideal of $A$.
	If $\gamma=0$, then $y^3=0$ in $A$ and hence $N^3=0$.

	(2) Assume that $\gamma \ne 0$. Let $\zeta\in \K$ be such that
	$\zeta^2+\zeta+1=0$. By Lemma~\ref{lem:invertible}, there exists an invertible
	element $x\in A$ such that $yx=\zeta xy$.
	According to Proposition~\ref{pro:radical}, there exists a primitive idempotent $e\in \K[y]$ 
	such that $\Rad A=\oplus_{i,j=0}^2x^i\Rad (eAe)x^j$. Let $\gamma\in \K\setminus\{0\}$ be such
	that $ye=\gamma e$.
  Then $eAe\simeq C(V,q_\gamma)$ by Lemma~\ref{lem:eKe}.
	Thus the claim on the semisimplicity of $A$
	follows from Theorem~\ref{thm:Clifford}.

	Assume that $q_\gamma $ is nondegenerate. Then $C(V,q_\gamma)$ is simple
	by Theorem~\ref{thm:Clifford}(2). Hence $A$ is simple by
	Proposition~\ref{pro:radical}(2). Since $\dim A=36$ by
	Theorem~\ref{thm:auxiliar}(1), we conclude that $A\simeq \K^{6\times 6}$.
\end{proof}

\section{The Nichols algebra of dimension $72$}
\label{section:dim72}

Again we assume that $\K$ is an algebraically closed field of characteristic different from
$2$ and $3$.
In this section we study the algebra $B$ presented by generators $a,b,c,d$ and
relations
\begin{align*}
	&a^2=b^2=c^2=d^2=0,\\
	&ab+bc+ca=ac+cd+da=ad+ba+db=bd+cb+dc=0,\\
  &(a+b+c)^6=0.
\end{align*}


Based on computer calculations it is known that $\dim B=72$ and that the
Hilbert series of $B$ is
\[
H(t)=1+4t+8t^2+11t^3+12t^4+12t^5+11t^6+8t^7+4t^8+t^9,
\]
see~\cite{MR1800709}.  In Theorem~\ref{thm:T:basis} we will prove these facts
by different methods.

\begin{defn}
	\label{def:tildeT}
	For any $\alpha_1,\alpha_2\in \K$, let $\defT(\alpha_1,\alpha_2)$ be
	the $\K$-algebra given by generators $a,b,c,d$ and relations
	\begin{align*}
		&a^2 -\alpha_1=b^2 -\alpha_1=c^2 -\alpha_1=d^2 - \alpha_1=0,\\
		&ca + bc + ab - \alpha_2=da + cd + ac - \alpha_2=0,\\
		&db + ba + ad - \alpha_2=dc + cb + bd - \alpha_2=0.
	\end{align*}
  Let $y=ac+cb+ba-\alpha_2$.
\end{defn}

\begin{rem}
	\label{rem:G_X}
	Let $G$ be the group given by generators $g_a,g_b,g_c,g_d$ with relations 
	\begin{align*}
		g_ag_b = g_bg_c = g_cg_a, && g_ag_c = g_cg_d = g_dg_a, \\
		g_ag_d = g_dg_b = g_bg_a, && g_bg_d = g_dg_c = g_cg_b.
	\end{align*}
	It is known that $G$ is a central extension of $\mathrm{SL}(2,3)$,
	see~\cite{MR3276225}.

	There is a unique $\K G$-module algebra
	structure on $\defT(\alpha_1,\alpha_2)$ such that $g_a,g_b,g_c,g_d$ 
	act on the generators $a,b,c,d$ according to Table~\ref{tab:X}.
	\begin{table}[h]
	\caption{The action of $G$.}
	\begin{tabular}{c|cccc}
		& $a$ & $b$ & $c$ & $d$\tabularnewline
		\hline 
		$g_a$ & $-a$ & $-c$ & $-d$ & $-b$\tabularnewline
		$g_b$ & $-d$ & $-b$ & $-a$ & $-c$\tabularnewline
		$g_c$ & $-b$ & $-d$ & $-c$ & $-a$\tabularnewline
		$g_d$ & $-c$ & $-a$ & $-b$ & $-d$\tabularnewline
	\end{tabular}
	\label{tab:X}
	\end{table}
\end{rem}

\begin{rem}
	The usual presentation for $B$ found in the literature involves a
	different degree-six relation. One computes
	\[
		(cb)^2=c(bc)b=c(-ab-ca)b=0.
	\]
	Then acting on this with $g_d$ and $g_d^2$ one obtains $(ac)^2=(ba)^2=0$. Now
	a direct computation shows that 
	\[
		(cb+ba+ac)^3=(cba)^2+(bac)^2+(acb)^2.
	\]
\end{rem}

\begin{rem}
	The algebra $B$ together with an appropriate comultiplication,
        counit and antipode
        is the Nichols algebra associated with the rack given by a
	conjugacy class of $3$-cycles in the alternating group $\Alt_4$
	and constant cocycle $-1$. It was found by Gra\~na in~\cite{MR1800709} and
	later used by Ngakeu, Majid and Lambert in noncommutative
	geometry~\cite{MR1904729}.  The group $G$ of Remark~\ref{rem:G_X} is the
	enveloping group of this rack.
\end{rem}

\begin{lem} 
	\label{lem:Tyrels}
	Let $\alpha_1,\alpha_2\in \K$.
	Then $y x=-(g_{d}\cdot x)y$ in $\defT(\alpha_1,\alpha_2)$
	for all $x\in \{a,b,c,d\}$.
\end{lem}

\begin{proof}
	Let $X=\{a,b,c,d\}$. 
	The element $g_d^2\in G$, where $G$ is the group
        in Remark~\ref{rem:G_X}, permutes the elements $b,c,d$ cyclically
        and fixes $y$.
	Therefore it suffices to prove that
  $$ yd=dy,\quad yb=ay.$$
	Equation $yd=dy$ is proved as follows.
	\begin{align*}
		yd=&(ba+ac+cb-\alpha_2)d\\
		=&b(\alpha_2-db-ba)+a(\alpha_2-da-ac)+c(\alpha_2-dc-cb)-\alpha_2d\\
		=& (\alpha_2-bd)b+(\alpha_2-ad)a+(\alpha_2-cd)c-\alpha_2d-\alpha_1(a+b+c)\\
		=&(dc+cb)b+(db+ba)a+(da+ac)c-\alpha_2d-\alpha_1(a+b+c)\\
		=&dy.
	\end{align*}
	Equation $yb=cy$ is obtained by the following steps:
	\begin{align*}
		yb=&(ba+ac+cb-\alpha_2)b\\
		=&b(ab-\alpha_2)+acb+\alpha_1c\\
		=&b(-ca-bc)+a(cb+ac)\\
		=&-\alpha_1c-(\alpha_2-ca-ab)a+a(cb+ac)\\
		=&a(-\alpha_2+ba+cb+ac)\\
		=&ay.
	\end{align*}
	This completes the proof.
\end{proof}

\begin{defn} \label{def:defT}
  For any $\alpha_1,\alpha_2,\alpha_3\in \K$, let
	$\defT(\alpha_1,\alpha_2,\alpha_3)$ be the deformation of $B$ given by
	generators $a,b,c,d$ and relations
	\begin{align*}
		&a^2 - \alpha_1=b^2 - \alpha_1=c^2 - \alpha_1=d^2 - \alpha_1=0,\\
		&ca + bc + ab - \alpha_2=da + cd + ac - \alpha_2=0,\\
		&db + ba + ad - \alpha_2=dc + cb + bd - \alpha_2=0,\\
		&(cb+ba+ac-\alpha_2)^3-\alpha_3=0.
	\end{align*}
\end{defn}

\begin{rem}
	\label{rem:bimodule}
	Let $\alpha_1,\alpha_2,\alpha_3\in\K$. The algebra $\defT(\alpha_1,\alpha_2,\alpha_3)$ 
	is naturally a $\defIA(\alpha_1,\alpha_2,\alpha_3)$-bimodule where the action of $a,b,c$ and $y$ 
	is given by multiplication with $a,b,c$ and $y$, respectively.
\end{rem}

\begin{pro}
	\label{pro:Ore}
  For any $\alpha_1,\alpha_2,\alpha_3\in \K$,
	the algebra $\defT(\alpha_1,\alpha_2,\alpha_3)$ is isomorphic to the Ore extension 
	$\defIA(\alpha_1,\alpha_2,\alpha_3)[d;\partial,\sigma]/(d^2-\alpha_1)$, where 
	\[
		\sigma\in\Aut(\defIA(\alpha_1,\alpha_2,\alpha_3)),\quad
		\sigma(a)=-c,
		\quad
		\sigma(b)=-a,
		\quad
		\sigma(c)=-b,
	\]
	and the $(\sigma,\id)$-skew derivation $\partial$ of
	$\defIA(\alpha_1,\alpha_2,\alpha_3)$ is given by 
	\[
		\partial(a)=\alpha_2-ac,
		\quad
		\partial(b)=\alpha_2-ba,
		\quad
		\partial(c)=\alpha_2-cb.
	\]
\end{pro}

\begin{proof}
	It is straightforward to prove that the automorphism $\sigma$ and the skew
	derivation $\partial$ of $\defIA(\alpha_1,\alpha_2,\alpha_3)$ exist. For
	example:
	\begin{align*}
		\partial(a^2)=\partial(a)a+\sigma(a)\partial(a)=(\alpha_2-ac)a-c(\alpha_2-ac)=0.
	\end{align*}
	The rest follows from the definitions of $\defIA(\alpha_1,\alpha_2,\alpha_3)$
	and $\defT(\alpha_1,\alpha_2,\alpha_3)$.
\end{proof}

Recall from Definition~\ref{def:tildeT} that $y=ac+cb+ba-\alpha_2$. 

\begin{thm}
	\label{thm:T:basis}
	For any $\alpha_1,\alpha_2,\alpha_3\in\K$ the algebra 
	$\defT(\alpha_1,\alpha_2,\alpha_3)$ is a PBW deformation of $B$ and 
	\[
	(b-a)^{n_1}a^{n_2}c^{n_3}y^{n_4}d^{n_5},\quad n_1,n_4\in\{0,1,2\},n_2,n_3,n_5\in \{0,1\},
	\]
	is a basis of $\defT(\alpha_1,\alpha_2,\alpha_3)$. 
\end{thm}

\begin{proof}
	First one checks that $(d^2-\alpha_1)a=b(d^2-\alpha_1)$. By acting on this
	equation with $g_d\in G$ it follows that the left ideal of
	$\defT(\alpha_1,\alpha_2,\alpha_3)$ generated by $d^2-\alpha_1$ is a
	two-sided ideal. Hence 
	\begin{equation}
		\label{eq:T:decomposition}
		\defT(\alpha_1,\alpha_2,\alpha_3)\simeq\defIA(\alpha_1,\alpha_2,\alpha_3)\otimes\K[d]/(d^2-\alpha_1)
	\end{equation}
	as a left module over $\defIA(\alpha_1,\alpha_2,\alpha_3)$ by
	Proposition~\ref{pro:Ore}, see Remark~\ref{rem:bimodule}.  Now apply
	Theorem~\ref{thm:auxiliar}(1) to obtain the claimed basis of
	$\defT(\alpha_1,\alpha_2,\alpha_3)$. 
	Hence $\dim\defT(\alpha_1,\alpha_2,\alpha_3)=\dim B=72$ for all
	$\alpha_1,\alpha_2,\alpha_3\in\K$.  Therefore
	$\defT(\alpha_1,\alpha_2,\alpha_3)$ is a PBW deformation of $B$.
\end{proof}

\begin{rem}
  In view of Theorem~\ref{thm:T:basis} it is reasonable to ask for
  the defining relations of $\defT(\alpha_1,\alpha_2,\alpha_3)$
  in terms of the generators $b-a,a,c,y,d$. By translating the nine defining
  relations in Definition~\ref{def:defT} one obtains the relations
\begin{align*}
  \begin{aligned}
  a^2&=c^2=d^2=\alpha_1,\\
  a(b-a)&=-(b-a)a-(b-a)^2,\\
  ca&=-ac-(b-a)c+(b-a)a+(b-a)^2+\alpha_2-\alpha_1,\\
  da&=-ac-cd+\alpha_2,\\
  d(b-a)&=cd-ad+ac-(b-a)a-\alpha_1,\\
  y&=c(b-a)-(b-a)c+2(b-a)a+(b-a)^2,\\
  y^3&=\alpha_1
  \end{aligned}
\end{align*}
  (which is not something we prefer to work with).
  Nevertheless, the theorem implies directly that the Hilbert series of $B$ is
  the polynomial
  $$H(t)=(1+t)^3(1+t+t^2)(1+t^2+t^4).$$
\end{rem}

\begin{thm}
	\label{thm:T:semisimple}
	The algebra $\defT(\alpha_1,\alpha_2,\alpha_3)$ is semisimple if and only if
	\[
	\alpha_3(\alpha_3+(\alpha_1+\alpha_2)(3\alpha_1-\alpha_2)^2)\ne0.
	\]
	In this case
	$\defT(\alpha_1,\alpha_2,\alpha_3)\simeq\left(\K^{6\times6}\right)^2$.
\end{thm}

\begin{proof}
	Assume that $\alpha_3=0$. Lemma~\ref{lem:Tyrels} implies that the left ideal
	generated by $y$ is a two-sided nilpotent ideal. Hence
	$\defT(\alpha_1,\alpha_2,0)$ is not semisimple. 
	
	Now assume that $\alpha_3\ne0$. Let $\zeta\in\K$ be such that
	$\zeta^2+\zeta+1=0$. Lemma~\ref{lem:invertible}
	and \eqref{eq:u1v1-1}
	imply that there exists an invertible element $x\in
	\defT(\alpha_1,\alpha_2,\alpha_3)$ such that $yx=\zeta xy$.
	By Proposition~\ref{pro:radical}(3),
	there exists a primitive idempotent $e\in \K[y]$ such that
	$\defT(\alpha_1,\alpha_2,\alpha_3)$
	is semisimple if and only if
	$T=e\defT(\alpha_1,\alpha_2,\alpha_3)e$ is semisimple.
	Since $y^3=\alpha_3$, there exists $\gamma\in \K$ such that $ye=\gamma e$ and
	$\gamma^3=\alpha_3$.

	Recall that  $yd=dy$ by Lemma~\ref{lem:Tyrels}. Therefore
	$T$ is generated by $et$, $ev_+^3$ and $ed$
	by
	\eqref{eq:T:decomposition} and Lemma~\ref{lem:eKe}, where
	$t=a+b+c$ and $v_+=a+\zeta b+\zeta^2c$.
	Moreover, $\dim T=8$.
	Now we obtain that 
	\begin{align}
		\label{eq:d}
		d^2 &= \alpha_1,
		&&
		dt+td=2\alpha_2-y,
		&&
		dv_+^3+v_+^3d=2y^2+(3\alpha_1-\alpha_2)^2.
	\end{align}
	Hence $T$ is isomorphic to the Clifford algebra
	$C(W,q'_{\gamma})$, where $W$ is a three-dimensional
	vector space and $q'_{\gamma}$ is the quadratic form on $W$ given by 
	\begin{align*}
	q'_{\gamma} (\lambda_1x_1&+\lambda_2x_2+\lambda_3x_3)=(\gamma+3\alpha_1+2\alpha_2)\lambda_1^2
+(\gamma^3+\beta^3)\lambda_2^2+\alpha_1\lambda_3^2\\
&+(2\gamma^2-2\beta \gamma -\beta^2)\lambda_1\lambda_2
+(2\alpha_2-\gamma)\lambda_1\lambda_3
+(2\gamma^2+\beta^2)\lambda_2\lambda_3
	\end{align*}
	with respect to a fixed basis $x_1,x_2,x_3$ of $W$, and
	$\beta=3\alpha_1-\alpha_2$.
	The semisimplicity of $C(W,q'_{\gamma})$ is equivalent
	to the nondegeneracy of
	$q'_{\gamma}$, that is, to 
	$\alpha_3+(\alpha_1+\alpha_2)(3\alpha_1-\alpha_2)^2\ne0$.

	Assume now that $\defT(\alpha_1,\alpha_2,\alpha_3)$ (equivalently, $T$) is semisimple.
	By Proposition~\ref{pro:radical}(2) and Theorem~\ref{thm:Clifford}(3), the
	algebra
	$\defT(\alpha_1,\alpha_2,\alpha_3)$ is the direct product of two simple
	ideals. Then the claim follows from the fact that
	$72=6^2+6^2$ is the only decomposition of $72$ as a sum of two squares.
\end{proof}

A result analogous to Corollary~\ref{cor:FK3:pi} is the following:

\begin{cor}
	Any polynomial identity of $(\K^{6\times6})^2$ is a polynomial
	identity of $\defT(\alpha_1,\alpha_2,\alpha_3)$ for any $\alpha_1,\alpha_2,\alpha_3\in\K$. 
\end{cor}

Recall the definition of the group $G$ in Remark~\ref{rem:G_X}.  The following
result classifies PBW deformation of the algebra $B$. These deformations already
appeared in~\cite[Theorem 6.3]{MR3119229} in the context of the classification
of finite-dimensional pointed Hopf algebras with non-abelian coradical.

\begin{thm}
	\label{thm:defT}
	Each PBW deformation of $B$ in the category of $G$-modules is of the form
	$\defT(\alpha_1,\alpha_2,\alpha_3)$, $\alpha_1,\alpha_2,\alpha_3\in\K$.
\end{thm}

\begin{proof}
	Let $T$ be a PBW deformation of $B$ in the category of $G$-modules.
	Theorem~\ref{thm:T:basis} for $B$ implies that $T$
	is given by generators $a,b,c,d$ and relations
	\begin{gather*}
		a^2=t_1,\quad b^2=t_2,\quad c^2=t_3,\quad d^2=t_4,\\
		ca + bc + ab =t_5,\quad da + cd + ac =t_6,\\
		db + ba + ad =t_7,\quad dc + cb + bd =t_8,\\
		(cb+ba+ac)^3=t_9,
	\end{gather*}
	where $t_1,t_2,\dots,t_8$ are linear combinations of $1,a,b,c,d$, and
	$t_9$ is a linear combination of monomials of degree at most $5$ in the
	generators $a,b,c,d$.
	By acting with $g_a^3$ on the equation $a^2=t_1$ one obtains that
	$a^2=\alpha_1$ for some $\alpha\in \K$.
	By acting with $g_b$ and $g_c$ we conclude that
	$b^2=c^2=d^2=\alpha_1$. Similarly, the actions of $g_d^3$ and $g_b$, $g_c$ on the equation
	$ca+bc+ab=t_5$ imply that $t_5=t_6=t_7=t_8=\alpha_2$ for some $\alpha_2\in
	\K$.
	Let us rewrite the last defining relation of $T$ to be
  $$
		y^3=t'_9,
	$$
	where $y=cb+ba+ac-\alpha_2$. It remains to show that $t'_9\in\K1$. 

	Lemma~\ref{lem:span}
	and~\eqref{eq:T:decomposition} imply that 
	\[
		t'_9=\sum_{n_1,\dots,n_4\in\N_0}\lambda_{n_1,\dots,n_4}v_+^{n_1}t^{n_2}y^{n_3}d^{n_4}+
		\sum_{n_1,\dots,n_4\in\N_0}\mu_{n_1,\dots,n_4}v_-^{n_1}t^{n_2}y^{n_3}d^{n_4}
	\]
        for some $\lambda_{n_1,\dots,n_4},\mu_{n_1,\dots,n_4}\in \K$.
	Since $g_d\cdot y^3=y^3$ and 
	\begin{align*}
	g_d\cdot v_+^{n_1}t^{n_2}y^{n_3}d^{n_4}&=\zeta^{n_1}(-1)^{n_1+n_2+n_4}v_+^{n_1}t^{n_2}y^{n_3}d^{n_4},\\
	g_d\cdot v_-^{n_1}t^{n_2}y^{n_3}d^{n_4}&=\zeta^{-n_1}(-1)^{n_1+n_2+n_4}v_-^{n_1}t^{n_2}y^{n_3}d^{n_4},
	\end{align*}
	we conclude that $\lambda_{n_1,\dots,n_4}=\mu_{n_1,\dots,n_4}=0$ whenever
	$3\nmid n_1$ or $n_1+n_2+n_4$ is odd. Moreover the degree of $t'_9$ is at
	most five. Hence $d^2=\alpha_1$ and~\eqref{eq:u12} imply that 
	\[
		t'_9=\lambda_1+\lambda_2td+\lambda_3y+\lambda_4v_+^3t+\lambda_5v_+^3d+\lambda_6tyd+\lambda_7y^2
	\]
	for some $\lambda_1,\dots,\lambda_7\in\K$. Equations $dt=-td+2\alpha_2-y$
	and~\eqref{eq:u1v13} imply that 
	\[
		t'_9t-tt'_9=\left( (\beta^2\lambda_5-2\lambda_2t^2)+2(\beta\lambda_5-\lambda_6t^2)y-2\lambda_5y^2\right)d+t'',
	\]
	for some $t''\in\defIA(\alpha_1,\alpha_2)$. Since $y^3t=ty^3$,
	using~\eqref{eq:T:decomposition} and Theorem~\ref{thm:T:basis} we conclude
	that $\lambda_5=0$. Similarly, using~\eqref{eq:d} and $dy=yd$, we obtain that 
	\[
		t'_9d-dt'_9=(\lambda_6y^2+(\lambda_2-2\alpha_2\lambda_6)y-2\alpha_2\lambda_2)d+t'''
	\]
	for some $t'''\in\defIA(\alpha_1,\alpha_2)$. Since $y^3d=dy^3$, we conclude
	from Theorem~\ref{thm:T:basis} that $\lambda_2=\lambda_6=0$. Since 
	\[
		0=t'_9(t+2d)-(t+2d)t'_9=-6\lambda_4ty^2+\text{terms of degree $\leq3$},
	\]
	Theorem~\ref{thm:T:basis} implies that $\lambda_4=0$. Finally $y^3a=ay^3$ and
	Lemma~\ref{le:abcyrels} imply that 
	\[
		0=t'_9a-at'_9=\lambda_3(c-a)y+\lambda_7(b-a)y^2
	\]
	and hence $t'_9=\lambda_1$ by Theorem~\ref{thm:T:basis}.
\end{proof}


\bibliographystyle{abbrv}
\bibliography{refs}

\end{document}